\documentclass[letterpaper,12pt]{amsart}

\title{$K_{5,5}$ is fully reconstructible in $\CC^3$}

\author{Daniel Irving Bernstein}
\address{Tulane University Department of Mathematics}
\email{dbernstein1@tulane.edu}

\author{Steven J. Gortler}
\address{Harvard University School of Engineering and Applied Sciences}
\email{sjg@cs.harvard.edu}

\usepackage{latexsym,array,delarray,amsthm,amssymb,epsfig,mathtools}%eufrak
\usepackage[numbers]{natbib}
\usepackage{tikz,url}
\tikzstyle{vertex}=[circle, draw, inner sep=0pt,minimum size=6pt, fill=black]

\usetikzlibrary{calc}
\usepackage{verbatim}
\usepackage{graphicx}
\usepackage{centernot}

\theoremstyle{definition}
\newtheorem{theorem}{Theorem}[section]
\newtheorem{proposition}[theorem]{Proposition}
\newtheorem{lemma}[theorem]{Lemma}
\newtheorem*{lemma*}{Lemma}
\newtheorem*{lemma''*}{``Lemma''}

\newtheorem*{problem*}{Problem}
\newtheorem{definition}[theorem]{Definition}
\theoremstyle{remark}
\newtheorem*{claim*}{Claim}

\newcommand{\RR}{\mathbb{R}}

\newcommand{\CC}{\mathbb{C}}

\newcommand{\PP}{\mathbb{P}}

\def\p{{\bf p}}

\def\x{{\bf x}}

\usepackage[colorlinks=true]{hyperref}
\newcommand{\defn}[1]{\textcolor{blue}{\emph{#1}}}

\begin{document}

\begin{abstract}
    A graph $G$ is fully reconstructible in $\mathbb{C}^d$ if the graph is determined from its $d$-dimensional
    measurement variety. 
    The full reconstructibility problem has been solved for $d=1$ and $d=2$.
    For $d=3$, some necessary and some sufficient conditions are known and $K_{5,5}$ falls squarely within 
    the gap in the theory. In this paper, we show that $K_{5,5}$ is fully reconstructible in $\mathbb{C}^3$.
\end{abstract}\textbf{}

\maketitle

\section{Introduction}
Let $G$ be a graph with $n$ vertices and $m$ edges. 
Let $d$ denote
a fixed dimension. Associated with with $G$ is  its so-called
$d$-dimensional
measurement variety $M_{d,G} \subset \CC^m$.
We will define this below,
but roughly speaking, it represents all possible, squared edge length 
measurements,  over all $d$-dimensional configurations of $n$ points in $\CC^d$. 
We say that $G$ is fully reconstructible in $\CC^d$ if
% there is no other graph $H$ with the same $d$-dimensional
% measurement variety as $G$. This means that the information in 
$M_{d,G}$ fully determines $G$, as well as how
the edges in $G$ correspond to the coordinate axes in $\CC^m$, modulo graph isomorphism.

A natural question is to  characterize which graphs are fully
reconstructible in $\CC^d$. This and related questions
have been introduced and studied in~\cite{gtt,gj,ggj}.
One motivation for studying reconstructibility comes
from the realization problem of trying to determine a configuration of 
(generic) points given an unlabeled set of some of the 
point-pair distances (as well as the values of $d$ and $n$).
Characterizations for 
full reconstructibility are known for $d=1$ and $d=2$.
At present, some sufficient and some necessary conditions
for full reconstructibility are known for $d>2$, but nothing tight.
The complete bipartite graph, $K_{5,5}$ in $3$ dimensions falls squarely
into this gap, and so its resolution might help towards a 
better understanding of full reconstructibility.
In this paper, we show that $K_{5,5}$ is fully reconstructible in
$\CC^3$. 

\subsection{Definitions}

We summarize the necessary definitions from~\cite{ggj}.
\begin{definition}
We define a $d$-dimensional \defn{complex framework} to be a pair $(G,\p)$, where $G = (V,E)$ is a graph with $n$ vertices,
and $\p : V \rightarrow \CC^d$ is a complex mapping.
It is \defn{generic} if its coordinates  do not satisfy any algebraic equation with rational coefficients.
Given an edge $e = uv$ of $G$, its \textit{complex squared length} in $(G,\p)$ is 
\begin{equation*}
    m_{uv}(\p) = \left(\p(u) - \p(v)\right)^{T} \cdot \left(\p(u) - \p(v)\right) =  \sum_{k=1}^d (\p(u)_k - \p(v)_k)^2,
\end{equation*}
where $k$ indexes over the $d$ dimension-coordinates. Note that in this definition we do not use conjugation. For real frameworks, this coincides with the
usual (Euclidean) squared length.
We define  $m_{d,G}$,  
to be the function, 
$\CC^{nd} \rightarrow \CC^E$, with 
\begin{equation*}
m_{d,G}(\p) = \big(m_{uv}(\p)\big)_{uv \in E}.
\end{equation*}
\end{definition}

\begin{definition}
The \defn{$d$-dimensional measurement variety} of a graph $G$ (on $n$ vertices), denoted by $M_{d,G}$, is the Zariski-closure of $m_{d,G}(\CC^{nd})$.
\end{definition}
A measurement variety is irreducible.
Its dimension matches that of the generic rank of
the linearization of $m_{d,G}$.
If $G$ and $H$ are graphs, let $E(G)$ and $E(H)$ denote their edge sets and $V(G)$ and $V(H)$ their vertex sets. Given a bijection $\psi:E(G)\rightarrow E(H)$, there is a dual map $\psi':\mathbb{C}^{E(H)}\rightarrow \mathbb{C}^{E(G)}$ sending $f$ to $f\circ \psi$.
Since $M_{d,G}$ is a subset of $\mathbb{C}^{E(G)}$ and $M_{d,H}$ is a subset of $\mathbb{C}^{E(H)}$, we can ask if $M_{d,G} = \psi'(M_{d,H})$.
We say that an edge bijection $\psi:E(G)\rightarrow E(H)$ is \emph{induced by the graph isomorphism $\tau: V(G)\rightarrow V(H)$} if $\psi(uv) = \tau(u)\tau(v)$ for all edges $uv \in E(G)$.

\begin{definition}\label{def:graphrec}
A graph $G$ with no isolated vertices
is said to be \defn{fully reconstructible} in
$\CC^d$  if for all graphs $H$ with no isolated vertices and all edge bijections $\psi:E(G)\rightarrow E(H)$, if $M_{d,G} = \psi'(M_{d,H})$, then $\psi$ is induced by a graph isomorphism.
\end{definition}

In~\cite{ggj}, a different definition for full reconstructibility is given, but then, 
in their Theorem~2.18, 
they show this
to be equivalent to the definition we give above.

\begin{definition}
    Let $K_n$ denote the complete graph on $n$ vertices.
    The \defn{d-dimensional rigidity matroid} is the matroid of linear independence on the rows of the Jacobian matrix of $m_{d,K_n}$ at a generic point.
\end{definition}

The ground set of the $d$-dimensional rigidity matroid is in natural bijection with the edge set of $K_n$ and a graph $G$
on $n$ vertices is rigid in $d$-dimensions
iff its edge set is a spanning set of the $d$-dimensional rigidity matroid (i.e.~if the submatrix of the Jacobian (at a generic framework) 
of $m_{d,K_n}$,
restricted to the rows corresponding to edges of $G$, has rank $dn-\binom{d+1}{2}$).
We say that $G$ is $d$-rigid if it is generically rigid in $d$ dimensions, and that $G$ is a $d$-circuit if its edge set is a circuit of the $d$-dimensional rigidity matroid.
In what follows, we will make repeated tacit use of the following lemma.

\begin{lemma}\label{lem:dCircuit}
     $G$ is a $d$-circuit if and only if $M_{d,G}$ is a hypersurface and $M_{d,H}$ is full-dimensional in its ambient space when $H$ is a proper subgraph of $G$.
\end{lemma}
\begin{proof}
    This follows immediately from the fact that the $d$-dimensional rigidity matrix of $G$ is the Jacobian of $m_{d,G}$, that the row-submatrix corresponding to $H$ is the Jacobian of $m_{d,H}$, and that the rank of the Jacobian of a polynomial map is the dimension of the Zariski closure of its image.
\end{proof}

See also~\cite[Theorem~11]{rosen2020algebraic} for the algebraic statement and proof of the above lemma.

\subsection{Previous Results}
We summarize the known results on full reconstructibility
from~\cite{gtt, gj, ggj}. Most of the following conditions rely on various
notions from graph rigidity, which we will not review here;
all are defined in~\cite{ggj}.

For dimensions $1$ and $2$, full reconstructibility is
well characterized.

\begin{theorem}[{\cite[Theorem 2.20]{ggj}}]
    Let $G$ be a graph without isolated vertices.
    If $G$ has at least three vertices, then $G$ is fully reconstructible in $\mathbb{C}^1$ iff $G$ is $3$-connected.
    If $G$ has at least four vertices, then $G$ is fully reconstructible in $\mathbb{C}^2$ iff $G$ is generically globally rigid in two dimensions.
\end{theorem}

% \begin{theorem}[{\cite[Theorem 2.20]{ggj}}]
% Let $G$ be a graph on at least $3$ vertices and without isolated vertices. 
% Then $G$ is fully reconstructible in $\CC^1$ iff 
% $G$ is $3$-connected.
% \end{theorem}

% \begin{theorem}[{\cite[Theorem 2.20]{ggj}}]
% Let $G$ be a graph on at least $4$ vertices and without isolated vertices. 
% $G$ is fully reconstructible in $\CC^2$ iff 
% $G$ is generically globally rigid in $2$ dimensions.
% \end{theorem}

For general dimensions, we have the following necessary conditions.

\begin{theorem}[{\cite[Theorem 5.15]{gj} and~\cite[Theorem 3.7]{ggj}}]\label{thm:fullyrecmconnected}
Let $G$ be a graph without isolated vertices and suppose that $G$ is fully reconstructible in $\CC^d$. Then the $d$-dimensional rigidity matroid of $G$ is a connected matroid and $G$ is $3$-connected.
\end{theorem}
It is known that these conditions are not sufficient.
In particular it was shown in \cite[Example 4.6]{ggj} that there is a graph
that is $3$-connected, and has a connected rigidity 
matroid in $3$ dimensions, 
but is not  fully reconstructible in $\CC^3$. In fact, that graph is
also $4$-connected and
redundantly rigid in $3$ dimensions.
$K_{5,5}$ is $5$-connected and has a 
connected matroid in $3$ dimensions.
So it is a candidate for full reconstructibility.

For general dimensions, we have the following sufficient conditions.

\begin{theorem}[{\cite[Theorem 3.6]{ggj}}]\label{theorem:fullyreconstructible}
Let $d \geq 2$ and let $G$ be a graph on $n \geq d+2$ vertices that is generically globally rigid in $d$ dimensions. 
Then $G$ is fully reconstructible in $\CC^d$.
\end{theorem}
Generic global rigidity is known not to be sufficient for 
full reconstructibility.
In particular, one can create full reconstructibility
by a gluing over a small set.
\begin{theorem}[{\cite[Theorem 4.7]{ggj}}]\label{thm:gluingstronglyreconstructible}
Let $G = (V,E)$ be a graph with induced subgraphs $G_1 = (V_1,E_1)$ and $G_2 = (V_2,E_2)$ for which $V_1 \cup V_2 = V$ and $V_1 \cap V_2$ induces a connected subgraph of $G$ on at least three vertices. Let $d \geq 1$. If $G_1$ and $G_2$ are fully reconstructible rigid graphs on at least $d+1$ vertices in $\CC^d$,
then $G$ is fully reconstructible in $\CC^d$.
\end{theorem}

To date, all known fully reconstructible graphs are generically globally rigid, or built by this gluing construction.
On the other hand,  $K_{5,5}$ is neither generically globally
rigid in $3$ dimensions, nor obtainable by the above gluing construction.
Full reconstructibility of $K_{5,5}$ implies that these sufficient conditions are not necessary.

\section{Main Result}
The main result we will prove here is the following.

\begin{theorem}
\label{thm:main}
$K_{5,5}$ is fully reconstructible in $\CC^3$.
\end{theorem}

Its proof will involve three cases, two of which will
require a computer enumeration. The first case will deal
simply with edge permutations of $K_{5,5}$. 
The second case will deal with all other graphs on $10$ vertices.
The third case will cover all other graphs with a greater number
of vertices.

\subsection{Symmetries}
A \emph{permutation automorphism} of a variety $W \subseteq \mathbb{C}^E$ is a permutation of $E$ that sends $W$ to itself.
When $E$ is the edge set of a graph $G$, a permutation automorphism is \emph{trivial} if it is induced by a graph automorphism of $G$.
When $W$ has codimension one, then a permutation automorphism is
a symmetry in its single defining polynomial.

\begin{proposition}
\label{prop:1}
    $M_{3,K_{5,5}}$ has only has trivial permutation automorphisms.
\end{proposition}
\begin{proof}
    Since $K_{5,5}$ is a circuit of the 3d rigidity matroid, $M_{3,K_{5,5}}$ is a hypersurface.
    We will show that $M_{3,K_{5,5}}$ is the vanishing locus of $f$
    where $f:=\det(M)$ and $M$ is the following
    \[
        M:=\begin{pmatrix}
            0 & 1 & 1 & 1 & 1 & 1 \\
            1 & d_{05} & d_{06} & d_{07} & d_{08} & d_{09} \\
            1 & d_{15} & d_{16} & d_{17} & d_{18} & d_{19} \\
            1 & d_{25} & d_{26} & d_{27} & d_{28} & d_{29} \\
            1 & d_{35} & d_{36} & d_{37} & d_{38} & d_{39} \\
            1 & d_{45} & d_{46} & d_{47} & d_{48} & d_{49}
        \end{pmatrix}.
    \]
    Recall that $M_{3,K_{10}}$ is the vanishing locus of the $6\times 6$ minors of the $11\times 11$ Cayley-Menger matrix for ten points~\cite{borcea2004number}.
    Then, $M$ is a submatrix of this Cayley-Menger matrix and therefore $f:=\det(M)$ must vanish on $M_{3,K_{10}}$. Since $K_{5,5}$ is a subgraph of $K_{10}$, $M_{3,K_{5,5}}$ is a coordinate projection of $M_{3,K_{10}}$. Since $f$ only involves coordinates corresponding to edges of $K_{5,5}$, $f$ must also vanish on $M_{3,K_{5,5}}$.
    A computation in Macaulay2~\cite{M2} shows that $f$ is irreducible.
    Thus we now have that the vanishing locus of $f$ is an irreducible variety of the same dimension as $M_{3,K_{5,5}}$ containing $M_{3,K_{5,5}}$.
    Therefore the two varieties are equal.

    Let $\phi$ be a permutation automorphism of $M_{3,K_{5,5}}$, i.e.~a permutation of $E(K_{5,5})$, the set of coordinates of $M_{3,K_{5,5}}$, satisfying $f = f\circ \phi$.
    We show that $\phi$ is trivial.
    To begin, we claim that if $e,f \in E(K_{5,5})$ share a vertex, then so do $\phi(e)$ and $\phi(f)$.
    Indeed, the monomials of $f$ are products $d_{i_1j_1}d_{i_2j_2}d_{i_3j_3}d_{i_4j_4}$ where $\{i_1,i_2,i_3,i_4\}$ and $\{j_1,j_2,j_3,j_4\}$ are respectively four-element subsets of $\{0,1,2,3,4\}$ and $\{5,6,7,8,9\}$.
    That $\phi$ is a permutation automorphism of $M_{3,K_{5,5}}$ is equivalent to the condition that $f = f\circ \phi$.
    Let $\phi(M)$ denote the matrix obtained from $M$ by applying $\phi$ entrywise. Then $\det(\phi(M)) = f$. If $\phi(d_{ij}) = d_{lm}$ and $\phi(d_{ik}) = d_{np}$, then $l=n$ or $m=p$. Otherwise, $\phi(f)$ would have a monomial divisible by $d_{lm}d_{np}$.
    Applying the inverse of $\phi$ then shows that $f$ has a monomial divisible by $d_{ij}d_{ik}$, which is false, so the claim is proven.
    
    Since $K_{5,5}$ has no cliques of size greater than two, the only way a set of edges can satisfy the property that each pair shares a vertex is for that set of edges to form a star. Thus $\phi$ sends stars to stars. Define $\tau$ be the vertex permutation of $K_{5,5}$ such that $\tau(v)$ is the non-leaf vertex of the star that is the image of the edges incident to $v$.
    Then $\phi(uv) = \tau(u)\tau(v)$.
\end{proof}

\subsection{Other graphs with 10 vertices}

\begin{definition}
A \emph{stress} of a $d$-dimensional framework $(G,\p)$ is an element of the left kernel of the Jacobian of $m_{d,G}$ at $\p$.
A covector $\omega \in (\CC^E)^*$  is \defn{balanced}
if the sum of its coordinates equals $0$.
\end{definition}

\begin{definition}
Let $m$ be a smooth map between smooth manifolds $X$ and $Y$.
A point $\p$ in the domain $X$ of $m$ is a \defn{regular point}
if $\text{rank}\; dm_\p= \max_{u\in X} \text{rank}\; df_u$.
\end{definition}
Note that if $G$ 
is a $d$-circuit, then $(G,\p)$ has a unique (up to scale) stress if and
only if $\p$ is a regular point of $m_{d,G}$. At non-regular points,
it will gain a larger space of stresses.

The following can be deduced from~\cite{BR}.
\begin{lemma}
\label{lem:br}
Let $\p$ be a regular point of 
$m_{3,K_{5,5}}$.
Then the framework $(K_{5,5},\p)$ has a unique
equilibrium stress vector $\omega$, up to scale, and $\omega$
is balanced. 
\end{lemma}
\begin{proof}
From~\cite[Theorem 1]{BR} a 
framework of $K_{5,5}$ in three dimensions 
has exactly one stress (up to scale) that is balanced
unless one of the sets of $5$ points is
not in general affine position, in which case
it picks up a larger space of balanced stresses.
From~\cite[Theorem 6]{BR} any stress  of a 
framework of $K_{5,5}$ in three dimensions 
must be balanced unless the $10$ points are on a quadric,
in which case it picks up an unbalanced stress.

Thus regular points have exactly one stress and it is 
balanced.
\end{proof}

The following gives a useful geometric interpretation of 
equilibrium stress vectors.
\begin{lemma}\cite[Lemma 2.21]{ght}
\label{lem:snorm}
Let $\p$ be a regular point of the mapping
$m_{d,G}$, and let $m_{d,G}(\p)$ be a smooth point
of $M_{d,G}$.  Then the set of equilibrium stresses
for $(G,\p)$ equals the set of conormal vectors
to $M_{d,G}$ at $m_{d,G}(\p)$.  (A conormal vector is a covector
that annihilates every tangent vector.)
\end{lemma}

Let $G$ be a $3$-circuit.
Let $S(G)$ be the non-empty Zariski open subset of $M_{3,G}$ consisting of its smooth points.
Let $B(G)$ be the (possibly empty)
subset of $S(G)$ where the conormal vectors are balanced.

\begin{lemma}
\label{lem:sclosed}
Let $G$ be a $3$-circuit. Then $B(G)$ is a Zariski closed
subset of $S(G)$.
\end{lemma}
\begin{proof}
    Let the ``Gauss map'' be the map from $M_{3,K_{5,5}}$
    to $\PP^{24}$ that maps a smooth point $\x$ of
    $M_{3,K_{5,5}}$ to the conormal 
    of $M_{3,K_{5,5}}$ at $\x$. Its $25$ homogeneous coordinates are
    computed by evaluating, at $\x$, the $25$ partial derivatives
    of the single defining equation of $M_{3,K_{5,5}}$.
    Thus this map can be defined using polynomial functions.
    Meanwhile, the balanced-conormal subset is the pull back of
    a Zariski closed condition on $\PP^{24}$.
\end{proof}

\begin{lemma}
\label{lem:Beq}
$B(K_{5,5}) = S(K_{5,5})$.
\end{lemma}
\begin{proof}
    Let $\p$ be a a regular point of the  map
    $m_{3,K_{5,5}}$.
    From Lemma~\ref{lem:br} the  equilibrium stress of $(K_{5,5},\p)$ is balanced
    and the same is true in a neighborhood of $\p$.
    The image of this neighborhood under $m_{3,K_{5,5}}$ is 
    full dimensional in $M_{3,K_{5,5}}$. Removing the non-smooth points
    from this full dimensional image, we are left with a 
    full dimensional subset of $M_{3,K_{5,5}}$ where,
    from Lemma~\ref{lem:snorm}, the
    conormals of $M$ are balanced.
    Thus $\dim(B(K_{5,5}))=\dim(S(K_{5,5}))$. From Lemma~\ref{lem:sclosed}, $B(K_{5,5})$ is a closed subset of $S(K_{5,5})$.
    Since $M_{3,K_{5,5}}$ is irreducible this implies $B(K_{5,5})=S(K_{5,5})$.
    % Let $\p$ be a generic point in $\CC^{dn}$. Then $\p$ is a regular point of the  map
    % $m_{3,K_{5,5}}$ and $x:=m_{3,K_{5,5}}(\p)$ is a smooth point of $M:=M_{3,K_{5,5}}$. The same is true for some
    % neighborhood of $\p$.
    % From Lemma~\ref{lem:br} the  equilibrium stress of $(K_{5,5},\p)$ is balanced. From Lemma~\ref{lem:snorm}, the
    % conormal of $M$ at $x$ is balanced. Since $\p$ is generic, the same is true in a neighborhood of $x$ in $M$.
    % Thus $\dim(B(K_{5,5}))=\dim(S(K_{5,5}))$. From Lemma~\ref{lem:sclosed}, $B(K_{5,5})$ is a closed subset of $S(K_{5,5})$, which is irreducible,
    % so we must have set equality. 
\end{proof}

\begin{lemma}
\label{lem:Bless}
Let $G$ be a $3$-circuit. 
Let $(G,\p)$ be a framework in $\CC^3$ with unique (up to scale)
nonzero equilibrium stress $\omega$.
Suppose that $\omega$ is not balanced.
Then $\dim(B(G))<\dim(S(G))$. 
\end{lemma}
\begin{proof}
The set of regular points is open. 
Also the equilibrium
stresses change continuously as the configuration 
changes~\cite[Lemma B.4]{ghave}.
Thus, there is a neighborhood of $\p$ of regular points of $m_{3,G}$ where  
the equilibrium stresses are not balanced.
We can apply the same reasoning as the proof of Lemma~\ref{lem:Beq} to 
show that we have a full dimensional subset of $M_{3,G}$ where, the conormal vectors of $M$ are \emph{not} balanced.
From Lemma~\ref{lem:sclosed}, $B(G)$ is a Zariski closed subset of $S(G)$.
We have argued that $B(G)$ misses a full-dimensional subset of $S(G)$, and so,
from irreducibility of $M_{3,G}$, $B(G)$ must be of lower dimension.
\end{proof}

\begin{lemma}
\label{lem:nomatch}
Let $G$ be a $3$-circuit. 
Let $(G,\p)$ be a framework in $\CC^3$ with a unique (up to scale)
nonzero equilibrium stress $\omega$.
Suppose that $\omega$ is not balanced.
Then there is no edge bijection $\psi:E(K_{5,5}) \rightarrow E(G)$
under which $M_{3,K_{5,5}}=\psi'(M_{3,G})$.
\end{lemma}
\begin{proof}
This follows from Lemmas~\ref{lem:Beq} and \ref{lem:Bless}.
\end{proof}
Note that the balanced stress condition can be tested
without enumerating the possible edge bijections between
$K_{5,5}$ and $G$.

The following lemma is a basic and well-known observation about $3$-rigid $3$-circuits.
Recall that a \emph{cut} of a connected graph $G$ is a set $S$ of vertices such that the graph obtained from $G$ by removing $S$ is disconnected, and that a connected graph is \emph{$3$-connected} if it has no cuts of size $1$ or $2$.

\begin{lemma}\label{lem:basicRigidityFacts}
    Let $G$ be a $3$-rigid $3$-circuit. Then
    \begin{enumerate}
        \item $G$ is $3$-connected,
        \item $G$ contains no proper subgraph isomorphic to $K_5$, and
        \item $G$ has minimum degree at least $4$.
    \end{enumerate}
\end{lemma}
\begin{proof}
    Given circuits $C_1,C_2$ of any matroid, $C_1 \subseteq C_2$ can only happen if $C_1 = C_2$. Thus since $K_5$ is a $3$-circuit, any graph properly containing $K_5$ cannot be a $3$-circuit.
    
    If $\{u,v\}$ is a cut of $G$, then in any generic $3$-dimensional real framework $(G,{\bf p})$, the connected components of $G\setminus \{u,v\}$ can be independently rotated around the line through ${\bf p}(u)$ and ${\bf p}(v)$. Thus if $G$ is not $3$-connected, it is not $3$-rigid.
    
    If a vertex $u$ has degree $3$ or less,
    then the lengths of the edges incident to $u$ can be independently perturbed.
    Indeed, in any real three-dimensional framework on $G$ where $u$ has degree $3$,
    the length of any edge incident to $G$ can be perturbed by removing that edge, moving $u$ slightly along 
    the circle in the plane normal to the line through $u$'s remaining neighbors, centered at the point where this line and plane meet
    then putting the removed edge back.
    This implies that there is no polynomial that these three edge lengths must satisfy.
    Passing to the differential, we see that there is no linear relation involving the corresponding rows in the rigidity matrix.
\end{proof}

\begin{proposition}
\label{prop:2}
    If $G$ is a graph on $10$ vertices with 
    $M_{3,G}= M_{3,K_{5,5}}$,
    then $G = K_{5,5}$.
\end{proposition}
\begin{proof}
Since $M_{3,G}$ is a hypersurface in $\mathbb{C}^{25}$ and all its coordinate projections are full-dimensional, $G$ must be a $3$-circuit with $25$ edges.
From a dimension count, a $3$-circuit with $10$ vertices and $25$
edges must also be $3$-rigid. Via Lemma~\ref{lem:nomatch}, it is enough to show that for each $3$-rigid $3$-circuit $G$, either $G = K_{5,5}$ or the unique (up to scale) stress of $G$ is not balanced, in which case 
there is no edge bijection between $K_{5,5}$ and $G$
so that $M_{3,K_{5,5}}=M_{3,G}$.
We do this computationally in the following way.

We begin by using the \verb|geng| command of Nauty/Traces to generate the set $\mathcal{G}$ of all $2$-connected graphs with 25 edges and 10 vertices with minimum degree 4.
Lemma~\ref{lem:basicRigidityFacts} implies that all $3$-rigid $3$-circuits are contained in $\mathcal{G}$.
We randomly generate a configuration of $10$ points in $\RR^{3}$ and compute the rigidity matrix $M$ of $K_{10}$ for this configuration.
Then, for each $G \in \mathcal{G}$, we construct the matrix $M_G$ by restricting $M$ to the rows corresponding to edges of $G$. Let $\mathcal{G}'\subseteq \mathcal{G}$ consist of the graphs for which $M_{G}$ has rank $24$.
We then use Lemma~\ref{lem:basicRigidityFacts} to certify that each $G\in \mathcal{G}\setminus \mathcal{G}'$ is not a $3$-rigid $3$-circuit by showing that each such graph either contains a $K_5$ subgraph or fails to be $3$-connected.
Thus $\mathcal{G}'$ contains all $3$-rigid $3$-circuits with $10$ vertices and $25$ edges.

Every graph in $\mathcal{G}'$ has generic frameworks with exactly one nonzero stress (up to scale).
We then certify that if $G \in \mathcal{G}'\setminus\{K_{5,5}\}$, then stresses of $G$'s generic frameworks are not balanced by summing the entries of a stress of a framework on $G$ for each $G \in \mathcal{G}'$ and certifying that the only graph for which this quantity is $0$ is $K_{5,5}$.
A Mathematica script running these computations can be found at \url{https://dibernstein.github.io/Supplementary_materials/K55.html}.
\end{proof}

\subsection{Graphs with more vertices}
In order to show that $K_{5,5}$ is fully reconstructible
in $\CC^3$, it remains to show that no other graph with 25 edges has an isomorphic measurement variety. Any other such graph must be a circuit in the 3d rigidity matroid since its measurement variety is a hypersurface whose coordinate projections are all full-dimensional.

\begin{proposition}
\label{prop:3}
    Let $G$ be a $3$-circuit with $25$ edges. Then $G$ has $10$ vertices.
\end{proposition}
\begin{proof}
For a graph with $n\ge 4$ vertices, 
the rank of its $3$-dimensional rigidity matrix is at most
$3n-6$.  Thus a $3$-circuit, which must have exactly one stress,
    must have at least 10 vertices. Moreover, since all circuits of the 3d rigidity matroid have minimum degree at least 4, no $3$-circuit with 25 edges can have 13 or more vertices. 
    %So it now suffices to show that any 3d rigidity circuit with 25 edges on 1 11, or 12 vertices has a measurement variety that is not isomorphic to the measurement variety of $K_{5,5}$. 
    
    We check computationally that there are no circuits of the 3d rigidity matroid with 25 edges on 11 or 12 vertices. In particular, we  use the \verb geng  command of Nauty-Traces to generate all two-connected graphs on 11 and 12 vertices with 25 edges and minimum degree 4, then test which among them are circuits of the 3d rigidity matroid.
    
    As it turns out, none of them are, and we  certify this in the following way. We begin by eliminating all the graphs that are generically stress-free by certifying that they have a stress-free configuration.
    This leaves us with 1246 graphs. At this point, it becomes computationally feasible to eliminate all graphs with a $K_5$ subgraph. The remaining 113 graphs all fail to be circuits since they each contain a subgraph on six vertices with $3\cdot 6 - 5 = 18$ or more edges.
    A Mathematica script verifying this computation can be found at \url{https://dibernstein.github.io/Supplementary_materials/K55.html}.
\end{proof}

\begin{proof}[Proof of Theorem~\ref{thm:main}]
    If $G$ is a graph such that $M_{3,G} = M_{3,K_{5,5}}$, then Lemma~\ref{lem:dCircuit} implies that $G$ is a $3$-circuit. Proposition~\ref{prop:3} implies that $G$ has $10$ vertices and therefore Proposition~\ref{prop:2} implies $G = K_{5,5}$. Proposition~\ref{prop:1} then rules out the possibility of a nontrivial permutation automorphism.
\end{proof}

\section*{Acknowledgements}
This paper arose as part of the Fields Institute Thematic Program on Geometric constraint
systems, framework rigidity, and distance geometry. DIB was partially supported by US NSF grant DMS-1802902. SJG was partially supported by US NSF grant DMS-1564473.

\bibliographystyle{plain}
\footnotesize
\bibliography{reconstructability}

\end{document}